\documentclass[a4paper,12pt,oneside]{article}
\usepackage{amsmath,amssymb,amsthm,color}
\usepackage[margin=1in]{geometry}
\newcommand{\cdummy}{\cdot}

\newcommand{\tmop}[1]{\ensuremath{\operatorname{#1}}}

\definecolor{grey}{rgb}{0.75,0.75,0.75}
\definecolor{orange}{rgb}{1.0,0.5,0.5}
\definecolor{brown}{rgb}{0.5,0.25,0.0}
\definecolor{pink}{rgb}{1.0,0.5,0.5}
\numberwithin{equation}{section}
\newtheorem{theorem}{Theorem}[section]
\newtheorem{lemma}[theorem]{Lemma}

\theoremstyle{definition}
\newtheorem{definition}{Definition}[section]
\theoremstyle{remark}

\newcommand{\bke}[1]{\left ( #1 \right )}

\newcommand{\norm}[1]{ \| #1  \|}

\newcommand{\abs}[1]{\left | #1 \right |}

\def\de{\delta}

\def\e {\varepsilon}

\def\th{\theta}

\def\si{\sigma}

\def\ph{\varphi} %

\def\om{\omega}

\def\De{\Delta}

\def\La{\Lambda}

\def\Om{\Omega}

\newcommand{\R}{\mathbb{R}}

\renewcommand{\div}{\mathop{\rm div}}
\newcommand{\curl} {\mathop{\rm curl}}

\newcommand{\pd}{\partial}
\newcommand{\nb}{\nabla}
\newcommand{\td}{\tilde}

\renewcommand{\bar}[1]{\overline{#1}}
\newcommand{\lec}{{\ \lesssim \ }}

\newcommand{\cD}{\mathcal{D}}

\newcommand{\cX}{\mathcal{X}}

\newcommand{\I}{\infty}

\renewcommand{\[}{\begin{equation}}
\renewcommand{\]}{\end{equation}}

\newcommand{\donothing}[1]{}

\begin{document}

\title{Regularity criteria in weak $L^3$ for 3D incompressible
  Navier-Stokes equations}

\author{Yuwen Luo \and Tai-Peng Tsai}

\date{}%
 \maketitle

\begin{abstract}
We study the regularity of a distributional solution $(u,p)$ of the 3D
incompressible evolution Navier-Stokes equations. 
Let $B_r$ denote concentric balls in $\R^3$ with radius $r$. 
We will show that if $p\in L^{m} (0,1; L^1(B_2))$, $m>2$, and if
$u$ is sufficiently
small in $L^{\infty} (0,1; L^{3,\infty}(B_2))$, without any assumption on its gradient, then $u$ is bounded
  in $B_1\times (\frac{1}{10},1)$.  It is an endpoint case of the usual
  Serrin-type regularity criteria, and extends the steady-state result
  of Kim-Kozono to the time dependent setting. In the appendix we also show
some nonendpoint borderline regularity criteria.


 \end{abstract}

\section{Introduction}

This paper is concerned with the regularity of a distributional
solution $(u,p)$ of the 3D incompressible Navier-Stokes equations
\begin{equation}
\label{NS}
\pd_t u - \De u + (u \cdot \nb) u  + \nb p =0, \quad
\div u = 0.
\end{equation}
Denote $B_{r}=\{x\in \R^{3}: |x|<r\}$. 
Our goal is to prove interior regularity (i.e.~boundedness) of $u$
assuming that  $p\in L^{m} (0,1; L^1(B_2))$, $m>2$, and that
$u$ is sufficiently
small in $L^{\infty} (0,1; L^{3,\infty}(B_2))$.

\begin{definition}
A pair  $u \in L^{2} (  B_2\times (0, 1);\R^3 )$ and $ p \in L^{1}(B_2\times (0, 1))$
is a {\it distributional solution} of \eqref{NS} in $B_2 \times (0,1)$ if 
\begin{equation}
\label{dis.sol.1}
\int _0^1 \int_{B_2} \bke{ u\cdot (- \pd_t \zeta  - \De \zeta) - \sum_{i,j}u_i u_j \pd_i \zeta_j
- p \sum_{i}\pd_i \zeta_i}\,dx \,dt=0
\end{equation}
for any $\zeta \in C^2_c(B_2 \times (0,1);\R^3)$, and
\begin{equation}
\label{dis.sol.2}
 \int_{B_2}  u(x,t)\cdot \nb \phi(x) \,dx=0, \quad \forall \phi\in C^1_c(B_2) ,
\end{equation}
for almost every $t \in (0,1)$.
\end{definition}

\begin{definition}
A vector field  $u \in L^{2} (  B_2\times (0, 1);\R^3 )$ 
is a {\it very weak solution} of \eqref{NS} in $B_2 \times (0,1)$ if it satisfies
\eqref{dis.sol.2}, and also \eqref{dis.sol.1} 
for any $\zeta \in C^2_c(B_2 \times (0,1);\R^3)$ with $\div \zeta =0$, so that the last term in \eqref{dis.sol.1} involving $p$ is absent. A {\it weak solution} $u$ is a very weak solution which further satisfies
$u \in L^\I(0,1;L^2(B_2))\cap  L^2(0,1;H^1(B_2))$.
\end{definition}

Note that the definitions do not involve any boundary or initial conditions. 
Also note that the second definition does not 
explicitly involve the pressure $p$. A distributional solution is necessarily a 
 very weak solution.

Recall that it is an open problem whether the initial value problem or
initial-boundary value problem of the Navier-Stokes system \eqref{NS}
has a global classical solution for smooth and localized initial data
(with zero boundary condition).  An important regularity criterion due to Serrin
{\cite{Serrin62}} states that if a weak solution $u$ satisfies the
condition
\begin{equation} u \in L^s \left( 0, T ; L^q \right) \quad\tmop{with} \quad 3 < q <
   \infty , \quad\frac{2}{s} + \frac{3}{q} < 1, \end{equation}
then $u$ is locally bounded. 
 The borderline cases $3/p + 2/q = 1$, $3 < p \le \I$, were proved by Ladyzhenskaya
\cite{Lady1967}, 
 Sohr \cite{Sohr83}, Giga \cite{Giga86}, and 
Struwe \cite{Struwe88} under various settings.
For the end point case $(q, s) = (3,\I)$, i.e., $u \in L^\I(0,T;L^3)$, partial results are available in 
\cite{MR1078355, Struwe88, vonWahl, KozonoSohr}, and the full case in $\R^3$ is resolved by
Escauriaza,  Seregin and   {\v{S}}ver{\'a}k \cite{ESS}. 
See \cite{GKT072} and its references for various regularity criteria in terms of scaled norms.

Attempts were made to replace Lebesgue spaces by Lorentz spaces $L^{q,r}$ in these regularity criteria. Recall $L^{q,\I}$ is the weak $L^q$ space.
Takahashi \cite{MR1078355} showed regularity of weak solutions assuming $\|u\|_{L^{s}L^{q,\infty}}$, $3<q\le \I$, is small enough. Chen and Price \cite{MR1874409} showed regularity at $(x_0,t_0)$ assuming $\sup_{B_r(x_0) \times (t_0-r^2,t_0)}|u(x,t)||x-x_0|^{1-\th}|t-t_0|^{\th/2}$ is sufficiently small for some $0<\th<1$ and $r>0$.   Sohr \cite{MR1877269} assumes $u\in L^{s,r}(0,T;L^{q,\infty})$  with $ 3<q<\infty$, $\frac{3}{q}+\frac{2}{s}=1$.
Kim and Kozono
\cite{KK04} proves interior regularity assuming  smallness of $\norm{u}_{ L^{s,\I}(0,T;L^{q,\infty})}$  with $ 3\le q<\infty$, $\frac{3}{q}+\frac{2}{s}=1$. It includes the end point $(q,s)=(3,\I)$.

All above-mentioned regularity criteria are for weak solutions with $u
\in L^\I(0,T;L^2)\cap L^2(0,T;H^1)$.  For {\it {distributional
    solutions}} assuming no gradient bound, the only known regularity results are for steady
states. The first type of results is the removability of singularity by
Dyer and Edmunds \cite{Dyer-Edmunds},  Shapiro \cite{MR0380158} and  Choe and Kim 
\cite{Choe-Kim}. In the most recent work \cite{Choe-Kim}, it is showed that if $(u,p)$ is a distributional solution in $B_1 \backslash \{ 0 \}$ and either $u(x)=o(|x|^{-1})$ as $x\to 0$  or  $u\in L^{3}(B_1)$, then $(u,p)$ is  a distributional solution in $B_1$.  
The second type of results is the regularity for distributional solutions $(u,p)$  in $B_1$. 
It is known in  \cite{Dyer-Edmunds, MR0380158,Choe-Kim} that $u$ is regular if $u \in L^{\beta}(B_1)$, $\beta>3$. Kim and Kozono \cite{KK06} 
shows the regularity assuming $(u,p)\in L^{3}_{loc}\times
L^{1}_{loc}$ or $(u,p)\in L^{3,\I}\times L^{1}_{loc}$ with $\norm{u}_{
  L^{3,\I}}$ sufficiently small. Also see Miura and Tsai \cite{MT12} which characterizes the asymptotes of a {\it very weak solution} $u$ in $B_1 \backslash \{ 0 \}$ with $\norm{|x| u(x)}_{L^\I}$ sufficiently small.

The main purpose of this article is to obtain a regularity criterion for distributional solutions of the time-dependent Navier-Stokes equations in the borderline class $L^{\infty}(0,1;L^{3,\infty}(B_2))$.
Our main result is the following.

\begin{theorem}
\label{thm1} There is a small constant $\e_1>0$ such that the following holds.
Suppose the pair $(u,p)$ is a distributional solution of the Navier-Stokes system \eqref{NS} in $B_2 \times (0,1)$, with $p \in L^m(0,1;
L^{1}(B_2))$ for some $m>2$, and 
\begin{equation}
\e = \norm{u}_{L^\I(0,1; L^{3,\I}(B_2))} \le \e_1.
\end{equation}
Then $u \in L^{\infty} \left( 
B_1 \times (\frac{1}{10}, 1) \right)$.
\end{theorem}

{\it Comments on Theorem \ref{thm1}}:
\begin{enumerate}
\item It is a borderline case of Serrin-type regularity criteria with
$s=\infty$ and $q=(3,\infty)$. 

\item We assume $p \in L^m(0,1; L^{1}(B_2))$, but we do not need it to be small.
Moreover, the small constant $\e_1$ is independent of $m$ and $\norm{p}_{L^m(0,1; L^{1}(B_2))}$. In addition, we make no assumption on the gradient of $u$.

\item The similar result in Kim-Kozono \cite{KK04} does not assume any thing on the pressure, but requires that $u$ is a weak solution, 
$u \in L^\I(0,1;L^2(B_2))\cap  L^2(0,1;H^1(B_2))$.

\item Our proof makes use of a subcritical interior regularity criterion for very weak solutions $u \in L^s(0,1;L^q(B_2))$, $3/q+2/s<1$, see Theorem \ref{thm6.1} in Appendix. It does not need any assumption on $\nb u$ or $p$.

\end{enumerate}

The main idea of its proof is as follows:
we first perform a cut-off and reformulate
the problem on the entire space $\R^3$. We next show the
existence of a more regular solution of the reformulated problem.  We finally show that the original solution must locally agree
with the newly constructed regular solution. Both existence are uniqueness are based on the linear estimate \eqref{est2} of Yamazaki \cite{Yamazaki}, which allows the time exponent to be $\I$.

Theorem \ref{thm1} can be considered an extension of the steady-state
result of Kim-Kozono \cite[Theorem 4]{KK06} to the time-dependent setting. We reformulate an important 3D case of \cite{KK06} below.

\begin{theorem}[Kim-Kozono]
\label{thm2}
There is a small constant $\e_2 >0$ with the following property.
Let $\Om$ be any open set in $\R^3$.
If $(u,p) \in  L^2_{loc}(\Om;\R^3) \times  L^1_{loc}(\Om)$ is a distributional solution
of the stationary Navier-Stokes equations \eqref{NS} in $\Om$ with zero force, and
if $u$ satisfies $\norm{u}_{L^{3,\I}(\Om)}\le \e_2$, then $u\in L^\I_{loc}(\Om)$.
\end{theorem}

We will show that Theorem \ref{thm2} is a corollary of Theorem \ref{thm1}.
It is worthy noting that our proof is different from that of
\cite{KK06}, and hence is a second (although not simpler) proof of Theorem \ref{thm2}.

Finally we present a modest improvement of Theorem \ref{thm2}.

\begin{theorem}
\label{thm3}
There is a small constant $\e_3 >0$ with the following property.
Let $\Om$ be any open set in $\R^3$.
If $u \in  L^2_{loc}(\Om;\R^3)$ is a very weak solution
of the stationary Navier-Stokes equations \eqref{NS} in $\Om$ with zero force, and
if $u$ satisfies $\norm{u}_{L^{3,\I}(\Om)}\le \e_3$, then $u\in L^\I_{loc}(\Om)$
and
\begin{equation}
\label{eq1.7}
\norm{u}_{L^\I(B(x_0,R))} \le \frac CR\, \norm{u}_{L^{3,\I}(\Om)}
\end{equation}
for any ball $B(x_0,2R)\subset \Om$, for a constant $C$ independent of $u$ and $R$.
\end{theorem}

The improvement is the explicit estimate \eqref{eq1.7}, 
the absence of any assumption on the pressure, and that
\eqref{dis.sol.1} is satisfied only for those test functions $\zeta$ with $\div \zeta=0$. It is based on an interior estimate without pressure assumption, due to
 {\v{S}}ver{\'a}k and Tsai \cite{ST00}, see Lemma \ref{th5.1}.
Note that a time-dependent version of Lemma \ref{th5.1} 
appears in \cite[Appendix]{CSTY08}, which however cannot be used to replace the distributional solution assumption in Theorem \ref{thm1} by very weak solution, since the exponent of  time integration in  \cite[Appendix]{CSTY08} has to be finite and cannot be $\I$.

The rest of the paper is structured as follows: In Sect.~\ref{sec2} we give a few 
results for the Stokes system. In Sect.~\ref{sec3} we prove Theorem \ref{thm1}.  In Sect.~\ref{sec4} we show Theorem \ref{thm2}.  In Sect.~\ref{sec5} we show Theorem \ref{thm3}. In the appendix we 
show a subcritical regularity criterion for very weak solutions which is used in the proof of 
Theorem \ref{thm1}, and also some borderline regularity criteria which are nonendpoint  analogue of Theorem \ref{thm1}.

\section{Preliminaries}
\label{sec2}

In this section we collect a few preliminary results. For $1\le p\le \I$ we denote by $p'$ its conjugate exponent, $\frac1{p'}=1-\frac1{p}$.

\subsection{Oseen's tensor}
We recall the fundamental solution of the Stokes system in $\R^3$ (the
Oseen's tensor, see \cite{Oseen} and \cite[page 235]{Solonnikov})
\begin{equation}
\label{Oseen-tensor}
S_{ij}(x,t)=\Gamma(x,t)\delta_{ij}+\frac{\partial^2}{\partial
x_i\partial x_j}\int_{\R^3}\frac{\Gamma(y,t)}{4\pi\abs{x-y}}dy,%
\end{equation}
where $\Gamma(x,t)=(4\pi t)^{-3/2}\exp(-|x|^2/4t)$ is the fundamental solution of the heat equation. It
is known in \cite[Theorem 1]{Solonnikov} that the tensor
$S=(S_{ij})$ satisfies the following estimates:
\begin{equation}\label{estimate-T}
\abs{D^\ell_x\pd^k_t S(x,t)}\leq C_{k,l}
(|x|+\sqrt{t})^{-3-\ell-2k}, \quad (\ell,k \ge 0),
\end{equation}
where $D^\ell_x$ indicates $\ell$-th order derivatives with respect
to the variable $x$.

A solution of the non-stationary Stokes system in $\R^{3}\times \R_+$,
\begin{equation}\label{Stokes}
\pd_t w - \De w+ \nb p =f + \nb\cdot F, \quad \div w =0, 
\end{equation}
with zero initial condition,
if $f=(f_j)$ and $F=(F_{jk})$ have sufficient decay, is given by
\begin{align}
w_i(x,t)&=\int_0^{t}\int_{\R^3}S_{ij}(x-y,s)
f_{j}(y,t-s)dyds \nonumber\\
 &\qquad -\int_0^{t}\int_{\R^3}(\partial_{k}S_{ij}(x-y,s))
F_{jk}(y,t-s)dyds. \label{Oseen-formula}
\end{align}
Here we have taken the convention $(\nb \cdot F)_i = \sum_{j} \pd_j F_{ij}$.

\subsection{Stokes flow in Lorentz spaces}
Let $L^{q,r}$ denote the usual Lorentz space for $q,r\in [1, \I]$. 
For their properties see for example \cite{BennetSharpley,BerghLofstrom}. Recall that $L^{q,q}=L^q$ and that $L^{q,\I}$ is also called weak $L^q$.

For $1<q<\infty$ and $1\le r\le \infty$, one has the Helmholtz decomposition
\begin{equation}
L^{q,r}(\R^3;\R^3) = L^{q,r}_\si(\R^3)\oplus G^{q,r}(\R^3)
\end{equation}
where
\begin{align*}
L^{q,r}_{\sigma}(\R^3)&=\{u\in L^{q,r}(\R^3;\R^3) |\quad \div u=0
\ {\rm in}\ \R^3\},
\\
G^{q,r}(\R^3)&=\{\nb p\in L^{q,r}(\R^3;\R^3) |\quad p \in L^1_{loc}(\R^3)
\}.
\end{align*}
Let $P$ denotes the Helmholtz projection operator from $L^{q,r}$ to $L^{q,r}_\sigma$
with respect to the Helmholtz decomposition. The Stokes operator $A=A_{q,r}$ on
$L^{q,r}_{\sigma}(\R^3)$ is defined by $A =
-P \Delta$ with domain 
\begin{equation}
D(A_{q,r}) = \{u \in L^{q,r}_{\sigma}(\R^3)|\quad \nb^2 u \in L^{q,r}(\R^3)
\}.
\end{equation}
Let
\begin{equation}
\cD = \{u \in C^2_c(\R^3;\R^3)|\quad \div u =0\}.
\end{equation}
It is dense in $L^{q,r}_\sigma(\R^3)$ for $1<q,r<\I$.

We will need the following estimates.

\begin{lemma}
\label{lem4} Let $\Om=\R^3$ and
suppose that $1 < p \le q < \infty$. There exist  constants $C_1=C_1(q)$ and 
$C_2=C(p,q)$ such that for every $v \in \cD$ and 
$u \in L^{p, 1}_{\sigma} ( \Om)$,
\begin{equation}
\label{est1}
\norm{\nb v}_{L^{q,1}( \Om)}\le C_1\norm{A^{1/2} v}_{L^{q,1}( \Om)},
\end{equation}
\begin{equation}\label{est2}
    \int_0^{\infty} t^{\frac3 {2 p} - \frac3 {2 q} - \frac 1 2} \left\| \nb e^{ - t
      A} u \right\|_{L^{q, 1}( \Om)} d t \leq C_2 \left\| u
    \right\|_{L^{p, 1}_\si( \Om)}.
\end{equation}
\end{lemma}

This lemma is also true if $\Om$ is a half space or a bounded domain in $\R^3$ with
smooth boundary. In fact, we have the estimates
\begin{equation}
\norm{\nb v}_{L^{q}( \Om)}\le C_{q,\Om}\norm{A^{1/2} v}_{L^{q}( \Om)},
\end{equation}
for all $1<q<\infty$ for $\R^3$, half-spaces and bounded domains, and for $1<q<3$ for exterior domains, see \cite[Theorem 3.6]{MR960838}, and \cite[(3.15)]{BorMiy90}. An interpolation gives \eqref{est1}. 
The borderline case $q=3$ of \eqref{est1} for exterior domains is proved by  Yamazaki \cite{Yamazaki}. 
Estimate  \eqref{est2} is proved by  Yamazaki \cite{Yamazaki} with  the restriction $q\leq 3$ since he uses \eqref{est1}. The same proof works for $\R^3$, half-spaces and bounded domains with $1<q<\I$.

\bigskip

With the help of Lemma \ref{lem4}, we can define the solution operator for 
the Stokes system
\begin{equation}
\pd_t v- \De v + \nb p = \nb \cdot F, \quad \div v=0, \quad v|_{t=0}=0
\end{equation}
in $\R^3 \times \R_+$, when $F$ is in $L^\I L^{s,\I}$. In this case \eqref{Oseen-formula} does not converge absolutely. %
Below $BC_w$ denotes
the class of bounded and weak-star continuous functions.

\begin{lemma}\label{th2.2} 
Fix $\frac 32<r<\I$. Let $s=\frac {3r}{r+3}\in (1,3)$, $r=s^{*}$. Define the linear operator  $\Phi  F$ by duality for $F=(F_{jk})_{j,k} \in L^\I (\R_+; L^{s,\infty}(\R^3))$:
\begin{equation}
\label{Phi1.def}
\left(( \Phi F)(t), \varphi \right) = \sum_{j, k = 1}^3
   \int_0^{t} \left( - F_{j k} \left( t - \tau, \cdummy \right),
   \partial_j \left( e^{ - \tau A}  \varphi
   \right)_k \right) d \tau    , \quad
\forall \varphi \in L^{r', 1}_{\sigma} ( \R^3 ), \forall t>0.
\end{equation}
Then
$\Phi  F \in BC_w([0,\I); L^{r,\I}_\sigma( \R^3 ))$, and for some $c=c(r)>0$,
\begin{equation}
\label{th2.2-eq2}
\norm{\Phi  F}_{L^\I L^{r,\I}} \lec \norm{ F}_{L^\I L^{s,\I}} .
\end{equation}
\end{lemma}

\begin{proof}
By Lemma \ref{lem4} with $(p,q)=(r',s')$, 
$\sup_t|(( \Phi F)(t), \varphi)|$ is bounded by $C \norm{ F}_{L^\I L^{s,\I}} \norm{\varphi}_{L^{r', 1}_{\sigma}}$ for any $\varphi \in L^{r', 1}_{\sigma} ( \R^3 )$, thus $\Phi F \in L^\I L^{r,\I}$ and we have \eqref{th2.2-eq2}.
Weak continuity can be shown by the same proof of \cite[Lemma 2.3]{KMT12}.
\end{proof}

\section{Proof of Theorem \ref{thm1}}
\label{sec3}
In this section we prove Theorem \ref{thm1}. We split the proof to
several steps. In \S\ref{S3.1}, we perform a cut-off and reformulate
the problem on the entire space $\R^3$. In \S\ref{S3.2}, we show the
existence of a regular solution of the reformulated problem.  In
\S\ref{S3.3}, we show that the original solution must locally agree
with the newly constructed regular solution.

We first show a better estimate of $p$.
Denote 
\begin{equation}
C_*= \norm{u}_{L^\I(0,1; L^{3,\I}(B_2))}+ \norm{p}_{L^m(0,1; L^1(B_2))}.
\end{equation}
Its first summand is small while the second may be large.
By taking the divergence of \eqref{NS}, $p$ is a distributional solution
of 
\begin{equation}
-\De p = \sum_{i,j} \pd_i \pd_j (u_i u_j).
\end{equation}
By the usual elliptic estimates, we have
\begin{equation}
\label{better.p}
\norm{p}_{L^m(0,1; L_x^{3/2,\I}(B_{1.9}))} 
\le C \norm{u}_{L^\I_t(0,1; L_x^{3,\I}(B_{2}))}^2 
+ C\norm{p}_{L^m (0,1; L_x^{1}(B_{2}))}\le C C_*.
\end{equation}

\subsection{Reformulation of the problem}
\label{S3.1}
In this subsection we perform a cut-off and reformulate the problem on
the entire space $\R^3$.

Let $\th(t)$ be a smooth cut-off function with $\th(t)=1$ for $t \ge
0.1$ and $\th(t)=0$ for $t< \frac 1{20}$. Let $\ph_0(x)$ be a smooth
cut-off function with $\ph_0(x)=1$ for $|x|\le 1$ and $\ph_0(x)=0$ for
$|x|\ge 5/4$.  Let $\ph(x,t) = \th(t) \ph_0(x)$.  Let $\td \ph(x)$ be
a smooth cut-off function so that $\td \ph(x)=1$ for $|x| \le 5/4$, and
$\td \ph(x)=0$ for $|x| \ge 3/2$.
Let%
\footnote{Although it is common to add a correction term $\hat u$ to the cut-off
$\ph u$ to make $\tilde u=\ph u + \hat u$ divergence-free, our
correction term $\hat u = \nabla \eta$ does not have compact support
as usual. It is because that we want $\hat u$ to be a potential so
that we can hide $\pd_t \hat u$ in $\nb \td p$, and hence need not
estimate it.  We have $|\nb \eta(x,t)|\lec |x|^{-2}$ for $|x|>2$, which is
sufficient for us. This technique has been used in, e.g.~\cite[(3.5)]{KMT12}. The term
$\De \eta$ in $\td p$ is not present in \cite{KMT12} since it is identically zero.}
\begin{equation}\label{tdu.def} 
\begin{split}
\td u &= \ph u + \nb \eta, \quad \eta = \frac 1{4\pi|x|} *_x ( \nb \ph
\cdot u ),%
\\
\td p &= \ph p - \pd_t \eta + \De \eta.
\end{split}
\end{equation}
Then $(v,q)=(\td u,\td p)$ satisfies 
$v|_{t=0}=0$
and
\begin{equation}%
\pd_t v - \De v + \nb q = f^0 + \nb \cdot F, \quad
\div v = 0
\end{equation}
in $\R^3 \times (0,1)$,
where
\begin{equation}
\label{f0.def}
f^0 =  u(\ph_t  +  \De \ph) + p \nb \ph + (\nb \ph\cdot u)u,
\end{equation}
\begin{equation}
F_{ij} = -2 (\pd_j \ph)u_i - \ph u_i u_j.
\end{equation}
We further single out the key quadratic term in $F$ and rewrite
\begin{equation}
- \ph u_i u_j = -\td \ph u_i  \ph u_j = \td \ph u_i(-v_j +\pd_j \eta).
\end{equation}
Then $F= f^1+f^2(v)$ with
\begin{align}
\label{f1.def}
f^1_{ij} &= -2 (\pd_j \ph)u_i +(\pd_j \eta)\td \ph u_i ,\\
\label{f2.def}
f^2_{ij}(v) &= - \td \ph u_i v_j  .
\end{align}

Summarizing, $(v,q)=(\td u,\td p)$ defined by \eqref{tdu.def} satisfies
\begin{equation}\label{v-eq0}
v|_{t=0}=0
\end{equation}
and
\begin{equation}\label{v-eq}
\pd_t v - \De v + \nb q = f^0 + \nb \cdot (f^1+f^2(v)) , \quad
\div v = 0,
\end{equation}
in $\R^3\times (0,1)$,
in the sense of distributions. We will treat
$u$ in the definitions of $f^k$ as known and $v$ as unknown.  Thus
$f^2(v)$ is linear in $v$.

We now take care of the source terms in \eqref{v-eq} and 
define $v^0 = (v^0_j)_{j=1}^3$ by \eqref{Oseen-formula},
\begin{align}
\label{v0.def}
(v^0)_i(x,t)&=\int_0^{t}\int_{\R^3}S_{ij}(x-y,s)
(f^0)_{j}(y,t-s)dyds \nonumber
\\
 &
\qquad -\int_0^{t}\int_{\R^3}(\partial_{k}S_{ij}(x-y,s))
(f^1)_{jk}(y,t-s)dyds. 
\end{align}
It solves 
\begin{equation}\label{v0-eq}
\pd_t v^0 - \De v^0 + \nb q^0 = f^0 + \nb \cdot f^1 , \quad
\div v^0 = 0.
\end{equation}
That is, it solves \eqref{v-eq} with $f^2$ removed.

Denote 
\begin{equation}
\cX^q = L^\I((0,1),L^{q,\I}). 
\end{equation}

\begin{lemma}
\label{th3.1}
For some $\de=\de(m)>0$,
we have $ v^0 \in
C([0,1]; L^{3,\I}_\si(\R^3)\cap L^{3+\de,\I}_\si(\R^3) )$ and $\norm{v^0}_{\cX^3\cap \cX^{3+\de}} \le C_m C_*$.
\end{lemma}

\begin{proof}
We first estimate $\nabla
  \eta$, which is bounded by $|x|^{-2}*|\nabla \varphi \cdummy u|$. Note $\nabla \varphi \cdummy u \in L^{r,\I}(\R^3)$ for any $1\le r\le 3$ since it has compact support.
For any $\frac 32 < q< \I$, choose $r\in (1,3)$ so that $\frac 1q=\frac 1r -\frac 13$ (i.e. $q=r^*$). By generalized
Young's inequality we have 
\begin{equation} 
\label{gdeta-est}
\| \nabla \eta \|_{L_{x}^{q,\I}} \leq C\left\| | x |^{-2}
\right\|_{L_{x}^{3 / 2,\infty}} \| \nabla \varphi \cdummy u
\|_{L_{x}^{r, \infty}}\le C 
\e, 
\end{equation}
which is uniform in $0<t<1$. 

For $f^0$ and $f^1$, since they have compact support,
by $u \in \cX^3$, $p \in L^m L^{3/2,\I} $ \eqref{better.p}, and \eqref{gdeta-est},
\begin{equation}
\label{f01-est}
\norm{f^0}_{L^m(0,1;L^1 \cap L^{3/2,\I})} + \norm{f^1}_{\cX^{r}} \le C_r C_*,
\quad
\forall r \in (1,3).
\end{equation}

We now estimate $v^0$.
By generalized Young's inequality, for $\de \ge 0$,
\[
\norm{v^0_i(t)}_{L^{3+\de,\I}} \leq c \int_0^t 
\|  S_{i j} (\cdot, s) \|_{L_{x}^a} \, g(t-s)
+ \|  \pd_k S_{i j} (\cdot, s) \|_{L_{x}^b} \| f^1 \|_{\cX^{r}}ds
\]
where $g(t)=\| f^0(\cdot,t) \|_{L_{x}^{3 / 2, \infty}}\in L^m_t(0,1)$,
 $\frac 1a= \frac 1{3+\de}-\frac 1{3/2}+1$ and 
$\frac 1b= \frac 1{3+\de}-\frac 1{r}+1$.
By \eqref{estimate-T} and $\norm{(|x|+\sqrt s)^{-k}}_{L^q_x}=C s^{\frac 3{2q}-\frac k2}$ for $k =3,4$ and $kq>3$,
\[
\norm{S_{ij}(\cdot,s)}_{L^a} \lec s^{3/(2a)-3/2},\quad
 \norm{ \pd_k S_{ij}(\cdot,s)}_{L^b} \lec s^{3/(2b)-2}.
\]
Thus
\[
\norm{v^0_i(t)}_{L^{3+\de,\I}} \leq c \int_0^t 
s^{-1+\frac{3}{2(3+\de)} } \, g(t-s)
+ s^{-\frac 12+ \frac 3{2(3+\de)} - \frac 3{2r}} \| f^1 \|_{\cX^{r}}ds.
\]
By Riesz potential estimates \cite[Lemma 7.12]{GilTru}, $\norm{v^0_i(t)}_{L^{3+\de,\I}} $
is uniformly bounded for $0<t<1$ if
\[
\frac 1m <  \frac 3{2(3+\de)}, \quad 
-\frac 12+ \frac 3{2(3+\de)} - \frac 3{2r}>-1.
\]
This is the case if $m>2$, $0 \le \de < \frac 32(m-2)$, and if we choose $r$ so that
$1<\frac 3r < \frac 3{3+\de}+1$.

The proof for the continuity in time is similar to that for heat potentials, and is omitted.
\end{proof}

Since $v^0$ satisfies \eqref{v0-eq},
Eqn.~\eqref{v-eq} is formally equivalent to
\begin{equation}
\label{v-int-eq}
v = v^0 - \Phi ( \tilde{\varphi} u \otimes v)
\end{equation}
for $0<t<1$,
where
 the operator $\Phi$ is defined in Lemma \ref{th2.2}.

\begin{lemma}
\label{th3.2}
The vector field $v=\td u$ defined by \eqref{tdu.def}, after redefinition 
on a set of time of measure zero, belongs to $
BC_w([0,1]; L^{3,\I}_\si(\R^3))$ with its norm bounded by $CC_*$. It
satisfies \eqref{v-int-eq} in $L^\I(0,1; L^{3,\I}_\si(\R^3))$.
\end{lemma}

\begin{proof}
Recall $v = \tilde{u} = \varphi u + \nabla \eta$. It is clear that
  $\varphi u \in L^{\infty} ( 0, 1 ; L^{3, \infty} (\R^3))$.  and so is $v$. 

Since $(u,p)$ is a
distributional solution of (\ref{NS}), $(v,q)=(\td u,\td p)$ is a distributional
solution of (\ref{v-eq}), or, for $w=v-v^0$, (recall $f^2 = \tilde{\varphi} u
\otimes v$)
\begin{equation}
\label{eq3.16}
- \iint w(\pd_t\ph +\De \ph) dxdt = -\iint f^2 : \nb \ph \,dxdt
\end{equation}
for any $\ph \in C^{2}_{c}(\R^3 \times (0,1))$ with $\div \ph=0$.
Taking  $\ph  = \theta (t) \eta(x)$ with $\theta (t) \in C_c^1 (0, 1)$ and 
$\eta \in \cD=C^2_{c,\sigma}(\R^3)$,
we get 
\begin{equation}
\label{eq3.24}
- \int (w,\eta)\th'(t) dt = \int \bke{(w,\De \eta)- (f^2,\nb \eta)}\th(t)  dt.
\end{equation}
Using $w \in L^\I L^{3,\I}$ and $f^2 \in L^{\I}L^{3/2,\I}$, the above equation
is valid for $\eta \in X$ where
\begin{equation}
X=\{\eta \in L^{3/2,1}_\si(\R^3): \nb^2 \eta \in L^{3/2,1}(\R^3), \nb \eta \in L^{3,1}(\R^3)\}.
\end{equation}
Eq.~\eqref{eq3.24} implies $\pd_t w \in L^\I(0,1; X^*)$. By redefining
$v(t)$ on a set of time of measure zero, we have $w(t) \in C([0,1];
X^*)$. Together with $w \in L^\I L^{3,\I}$, we get $w \in BC_w ([0,1];
L^{3,\I}_\si)$, and hence so is $v$.

For $0<t \leq t_1<1$, we can extend $\ph$ to the form $\varphi = \theta (t)
\psi (t, x)$, where $\theta (t) \in C_c^1 (0, t_1)$, $\psi (t, x) =
e^{- ( t_1 - t) A} \eta$, and $\eta \in \cD$. Using $\pd_t \psi +
\Delta \psi=0$, we get
\begin{equation}
\int_0^{\infty} - ( w(t), e^{- ( t_1 - t) A} \eta) \theta' ( t) dt =-
\int_0^{\infty}  ( f^2(t), \nabla e^{- (
t_1 - t) A} \eta)  \theta(t)\, dt .
\end{equation}
Using $w \in L^\I L^{3,\I}$ and $f^2 \in L^{\I}L^{3/2,\I}$, the above
equality is valid for $\eta \in L^{3,1}_\si(\R^3)$. As in
{\cite{KMT12}}, taking $\theta ( t) = \phi \left( \frac{t -
t_1}{\de} + 1 \right) - \phi \left( \frac{t}{\de}
\right)$ where $0 < \de \ll 1$, $\phi ( t) \in C^1_c (\R)$, $\phi ( t)
= 1$ for $t < 0$ and $\phi ( t) = 0$ for $t > 1$, then sending $\de
\rightarrow 0_+$, we have $\theta ( t) \rightarrow 1_{0 < t < t_1}$
and, by continuity of $( w ( t), e^{- ( t_1 - t)A} \eta)$,
\begin{equation} 
\int_0^{\infty} - ( w(t), e^{- ( t_1 - t) A} \eta) \theta' ( t) dt \to ( w (
     t_1), \eta) - ( w ( 0), e^{- t_1 A} \eta) = ( w ( t_1), \eta) . 
\end{equation}
Here we used the fact that $v (0) = 0$. Therefore
\begin{equation} ( w ( t_1), \eta) = -\int_0^{t_1} (
     f^2, \nabla e^{- ( t_1 - t) A} \eta)  dt. 
\end{equation}
Since $\eta \in L^{3, 1}_\si(\R^3)$ is arbitrary, $w = -\Phi (f^2)$ by definition.
\end{proof}

\subsection{Existence of regular solutions}
\label{S3.2}

In this subsection we prove the following existence lemma.
\begin{lemma}[Existence]\label{th3.3}
   Eq.~\eqref{v-int-eq} has a solution $v=\bar v $ in  $
  BC_w([0,1]; L_{\sigma}^{3, \infty}(\R^{3})\cap L_{\sigma}^{3+\de, \infty}(\R^{3}))$, where $\de>0$ is the small constant in Lemma \ref{th3.1}.
\end{lemma}

By Serrin-type subcritical regularity criteria for very weak solutions (see  Theorem \ref{thm6.1} in Appendix), $\bar v\in L^\I(\R^3 \times ( \tau, 1))$ for any $0<\tau<1$.

\begin{proof}
Let $Y=\cX^3 \cap \cX^{3+\de}$ with
$\norm{v}_Y=\norm{v}_{\cX^3}+\norm{v}_{\cX^{3+\de}}$.  By Lemma \ref{th3.1}
we have $v^0 \in Y$ with $\norm{v^0}_{Y} \le C_1 C_*$ for some $C_1=C_1(m)>0$.
For $v\in Y$, define
\begin{equation}
\La v = v^0 - \Phi (\tilde{\varphi} u \otimes v).
\end{equation}
We want to show that $\La$ is a contraction mapping in 
\begin{equation}
\label{eq3.30}
Y_1 = \{v \in Y:\quad \norm{v}_{Y}\le 2C_1C_* \}.
\end{equation}
First suppose $v \in Y_1$.  By Lemma \ref{th2.2} and H\"older
inequality for weak Lebesgue spaces \cite[page 15]{MR2445437}, for $s=\frac {3(3+\de)}{6+\de}$, $3+\de=s^*$,
\begin{align*}
\norm{\La v}_Y &\le \norm{v^0}_Y + \norm{\Phi (\tilde{\varphi} u \otimes v)}_Y
\\
&\le C_1 C_* + C \norm{\tilde{\varphi} u \otimes v}_{\cX ^{3/2} \cap \cX ^{s}}
\\
&\le C_1 C_* + C \norm{\tilde{\varphi} u}_{\cX ^{3}} \norm{ v}_{Y}
\le  C_1 C_* + C\e C_1 C_* \le 2 C_1 C_*,
\end{align*}
if $\e>0$ is sufficiently small.  This shows $\La v \in Y_1$.

Next we consider the difference: If $v_1,v_2 \in Y_1$, we have
\begin{align*}
\norm{\La v_1- \La v_2}_Y &= \norm{\Phi (\tilde{\varphi} u \otimes
  (v_1-v_2))}_Y \\ &\le C \norm{\tilde{\varphi} u \otimes
 (v_1-v_2) }_{\cX ^{3/2} \cap \cX ^{s}} \\ &\le C
\norm{\tilde{\varphi} u}_{\cX ^{3}} \norm{ v_1-v_2 }_{Y} \le C\e
\norm{v_1-v_2 }_{Y}.
\end{align*}
Therefore, if $\varepsilon$ is small enough (independent of $m$ and $C_*$), $\La$ is a contraction
mapping in $Y_1$ and has a unique fixed point $\bar v = \La \bar v$ in $Y_1$.
Since both $v^0 $ and $\Phi (\tilde{\varphi} u \otimes \bar v)$ are
weak-star continuous by Lemmas \ref{th2.2} and \ref{th3.1}, so is
$\bar v$.
\end{proof}

\subsection{Uniqueness}
\label{S3.3}

In this subsection we prove the following uniqueness lemma.
\begin{lemma}
\label{th3.4}
There is $\e_0>0$ such that, if $\e<\e_0$, and if $\bar{v},v\in
L^{\infty}(0,1;L_{\sigma}^{3,\infty}(\R^{3}))$ are two solutions of
\eqref{v-int-eq}, then $\bar{v}=v$.
\end{lemma}

\begin{proof}
Let $ w = v - \bar{v}$. It satisfies
$ w = \Phi \left( - \tilde{\varphi} u \otimes w \right)$ and hence
\begin{align*}
\norm{w}_{\cX^3} &= \norm{\Phi (\tilde{\varphi} u \otimes
  w)}_{\cX^3} \\ &\le C \norm{\tilde{\varphi} u \otimes
w }_{\cX ^{3/2} } \\ &\le C
\norm{\tilde{\varphi} u}_{\cX ^{3}} \norm{ w }_{ \cX ^{3}} \le C\e
\norm{w }_{\cX ^{3}}.
\end{align*}
Thus $w=0$ if $C\e<1$.  This proves Lemma \ref{th3.4}.
\end{proof}

\subsection{Conclusion of proof}
{\it{Proof of Theorem \ref{thm1}:}} By Lemma \ref{th3.3}, there exists
a regular solution $\bar v$ of \eqref{v-int-eq}, which coincides with
$\td u \in \cX^3$ by Lemma \ref{th3.4}. Hence $\td u$ is
regular. Since our distributional solution $u$ of (1.1) equals $\td u$
in $B_{1}\times(1/10,1)$, $u$ is regular in $
B_{1}\times(1/10,1)$. 
This proves Theorem \ref{thm1}.

\section{Proof of Theorem \ref{thm2}}
\label{sec4}
\begin{proof}
Let $(u,p)\in L^2_{loc}(\Om;\R^3) \times  L^1_{loc}(\Om) $ be a distributional
solution of the stationary Navier-Stokes equations in $\Om$ with  $\norm{u}_{L^{3,\I}(\Om)}\le \e \ll 1$.
For any $x_0 \in \Om$ we choose $R=R_x >0$ so that $B(x_0,2R)\subset \Om$. Define
\begin{equation}
v(x,t)=R u(x_0+Rx), \quad \pi(x,t)=R^2 p(x_0+Rx).
\end{equation}
 Then $(v,\pi)$ is a distributional solution of \eqref{NS} in $B_2 \times (0,1)$ with
trivial dependence on time and
\begin{equation}
\norm{v}_{L^\I(0,1;L^{3,\I}(B_2))} = \norm{u}_{L^{3,\I}(B(x_0,2R))}\le \e
\end{equation}
\begin{equation}
\norm{\pi}_{L^\I(0,1;L^{1}(B_2))} =R^{-1} \norm{p}_{L^{1}(B(x_0,2R))}< \I.
\end{equation}
By Theorem \ref{thm1}, $v$ is bounded in $B_1 \times (\frac 1{10},1)$ if $\e$ is sufficiently small. Thus
$u$ is bounded in $B(x_0,R)$. Since $x_0 \in \Om$ is arbitrary, we have shown $u \in L^\I_{loc}(\Om)$.
\end{proof}

\section{Proof of Theorem \ref{thm3}}
\label{sec5}

The proof of Theorem \ref{thm3} is based on the following interior estimate.

\begin{lemma}[Interior estimates for Stokes system]
\label{th5.1}
Let $B_{R_1} \subset B_{R_2} \subset \R^3$ be concentric balls with
$0<R_1<R_2$.  Assume that $v \in L^1_x(B_{R_2})$ is a very weak
solution of the Stokes system
\begin{equation}
- \Delta v_i + \pd_i p = \pd_j f_{ij}, \quad \div v = 0
\quad \text{in }B_{R_2},
\end{equation}
where $f_{ij} \in L^q(B_{R_2})$, $1<q<\infty$.
Then $v \in W^{1,q}_{loc}$, there is a $p \in L^q_{loc}$ so that the
above equation is satisfied in distribution sense and, for some constant
$c=c(q,R_1,R_2)$,
\begin{equation} \label{eqA1}
\norm{\nb v}_{L^q(B_{R_1})} + \inf_{a \in \R}
\norm{p-a}_{L^q(B_{R_1})} \le c \norm{f}_{L^q(B_{R_2})} + c
\norm{v}_{L^1(B_{R_2}\backslash B_{R_1})}.
\end{equation}
\end{lemma}

This lemma is \cite[Theorem 2.2]{ST00}. Although the statement in \cite{ST00} assumes
$v \in W^{1,q}_{loc}$, its proof only requires $v \in L^1$. 
 Similar estimates for the time-dependent Stokes
system appeared in \cite[Lemma A.2]{CSTY08}, and include Lemma \ref{th5.1} as a special  case.
An important feature of these estimates is that a bound of the
pressure $p$ is not needed in the right side.  This is desirable if we
want to study solutions for which we do not a priori have any
estimate of the pressure.

\medskip

\begin{proof}[Proof of Theorem \ref{thm3}]
Denote $\e = \norm{u}_{L^{3,\I}(\Om)}\le \e_3$.
For any $B(x_0,2R)\subset \Om$, 
 define
\begin{equation}
v(x,t)=R u(x_0+Rx).
\end{equation}
 Then 
\begin{equation}%
\norm{v}_{L^{3,\I}(B_2)} = \norm{u}_{L^{3,\I}(B(x_0,2R))} \le \e,
\end{equation}
and $v$ is a very weak solution of the Stokes system
\begin{equation}
- \Delta v_i + \pd_i p = \pd_j f_{ij}, \quad \div v = 0
\end{equation}
in $B_2$
with
\begin{equation}
f_{ij} = - v_i v_j.
\end{equation}

Fix any $1<q<3/2$. We have $v \in L^{3,\I}(B_{2}) \subset L^{2q}(B_{2})$ and
hence $\norm{f}_{L^q(B_{2})}\le C\norm{v}_{L^{2q}(B_{2})}^2 \le C \e^2$.
By Lemma \ref{th5.1}, $v \in W^{1,q}(B_{3/2})$ and there is a function $p\in L^{q}(B_{3/2})$ so that
\begin{equation}
\norm{\nb v}_{L^q(B_{3/2})} + \inf_{a \in \R}
\norm{p-a}_{L^q(B_{3/2})} \le c \norm{f}_{L^q(B_{2})} + c
\norm{v}_{L^1(B_{2})}\le C \e,
\end{equation}
and $(v,p)$ is a distributional solution of the stationary Navier-Stokes equations.
We may add a constant to $p$ so that the infimum of $\norm{p-a}_{L^q(B_{3/2})}$ occurs at $a=0$.
By Theorem \ref{thm2}, we conclude $v \in L^{\I}(B_1)$. Moreover,
by \eqref{eq3.30},
\begin{equation} 
\norm{v}_{L^{4}(B_{5/4})} \le C (\norm{v}_{L^{3,\I}(B_{3/2})} + \norm{p}_{L^{1}(B_{3/2})}) \le C \e.
\end{equation}
The usual bootstraping argument with small $\e$ gives $\norm{v}_{L^{\I}(B_{1})} \le C \e$ (compare the proof of Theorem \ref{thm6.1}).
Thus $|u|\le C\e /R$ in $B(x_0,R)$. 
\end{proof}

\section*{Appendix}
\setcounter{equation}{0}
\setcounter{theorem}{0}
\renewcommand{\theequation}{A.\arabic{equation}}
\renewcommand{\thetheorem}{A.\arabic{theorem}}

In the first part of this Appendix we prove a {\it subcritical} Serrin-type interior
regularity criteria for very weak solutions, which is used in the proof of Theorem \ref{thm1} . It assumes higher
integrability of $u$ than Theorem \ref{thm1}. However, it is concerned
with more general very weak solutions (than distributional solutions) and makes no assumption on the pressure.

\begin{theorem}
\label{thm6.1}
Suppose 
for some $q\in (3,\I]$, $s\in [3,\I]$, $3/q+2/s<1$, 
\[
u\in L^s(0,1;L^q(B_1))\cap  L^\I(0,1;L^1(B_1)) 
\]
is a very weak solution of \eqref{NS}, then $u\in
L^\I(\tau,1;L^{\I}(B_{1-\tau})) $ for any small $\tau>0$.
\end{theorem}

{\it Remark}. The second assumption $u\in L^\I(0,1;L^1(B_1)) $ is
necessary because of Serrin's example $u(x,t)=g(t) \nb h(x)$ for some
harmonic $h$, $\De h=0$. The condition $s\ge 3$ could be relaxed but the proof would be
more
tedious.

\begin{proof}
The weak form of \eqref{NS} can be considered as the weak form of the
inhomogeneous Stokes system
\[
\pd_t u - \De u + \nb p = \div F, \quad F_{ij} = -u_i u_j.
\]
Choose integer $K >0$ so that
\[
0<\si:=\frac 2{3K}< \frac 15(1-\frac 3q-\frac2s).
\]
Let $\de=\tau/(K+1)$ and $Q_k = B_{1-k\de} \times (k\de,1)$, $0 \le k \le K+1$.  Note $F
\in L^{s/2} L^{q/2}(Q_0)$.  By \cite[Lemma A.2]{CSTY08}, $\nb u$
exists and for any large $m<\I$,
\[
\norm{\nb u}_{ L^{s/2} L^{q/2}(Q_1)}\le c \norm{F}_{ L^{s/2} L^{q/2}(Q_0)} + \norm{u}_{L^\I L^{1}(Q_0)}<\I.
\]

Consider now
the vorticity $\om: =\curl u\in L^{p_0}(Q_1)$, $p_0 =  3/2$. It satisfies the inhomogeneous heat equation
\[
(\pd_t - \De )\om = \div G, \quad G_{ij}= u_i \om_j - u_j \om_i.
\]
 The same induction argument of
Serrin \cite{Serrin62} using potential estimate shows that
\[
\om \in L^{p_k}_{t,x}(Q_{k}), \quad \frac 1{p_k} = \frac
1{p_0}-k\si, \quad 0 \le k \le K,
\]
\[
G\in L^{a_k}_t L^{b_k}_x(Q_{k}),\quad \frac 1{a_k} = \frac
1{p_k}+\frac 1s\le 1, \quad \frac 1{b_k} = \frac 1{p_k}+\frac 1{q}\le 1.
\]
Note  $p_K=\I$ and $a_k \ge 1$, $b_k \ge 1$ thanks to $s,q \ge 3$.

The usual elliptic estimate \cite[Lemma A.1]{CSTY08}
gives
\[
\norm{\nb u}_{L^\I L^4(Q_{K+1})} \le c \norm{\om}_{L^\I
  L^4(Q_{K})} + c\norm{u}_{L^\I L^{1}(Q_{0}) }< \I.
\]
By Sobolev imbedding, $u \in L^\I(Q_{K+1})$.
\end{proof}

In the second part of this Appendix we prove a
nonendpoint borderline 
analogue of Theorem \ref{thm1}. It is not used in the rest of paper. Its proof is similar to
that of Theorem \ref{thm1} but is simpler: It uses the pointwise estimates \eqref{estimate-T}  of Oseen's tensor instead of Lemma \ref{th2.2}.

\begin{theorem}
Let $q\in (3,\I]$, $s\in [3,\I]$ with $3/q+2/s \leq 1$ and $m\ge1$,	$m >
\frac{2q}{3(q-2)}$. Suppose
\[
u\in L^s(0,1;L^q(B_1))\cap  L^\I(0,1;L^1(B_1)),\quad p\in L^{m}(0,1;L^{1}(B_{1})) 
\]
is a distributional solution of \eqref{NS} in $B_1 \times (0,1)$, then $u\in
L^\I(\tau,1;L^{\I}(B_{1-\tau})) $ for every small $\tau>0$.  
\end{theorem}
 
Note that all norms $\norm{u}_{L^s L^q}$, $\norm{u}_{L^\I L^1}$ and  $\norm{p}_{L^m L^1}$ need not be small. The condition $m >
\frac{2q}{3(q-2)}$ is implied by $m\ge1$ if $q>6$.

\begin{proof}
  By Theorem
  \ref{thm6.1}, it suffices to consider the case $3/q+2/s=1$. Since $q>3$ we have $s<\I$, and $\norm{u}_{L^s(t_0,t_1;L^q(B_1))}\to 0$ as $t_1-t_0\to0$, uniformly in $0\le t_0<t_1\le 1$. To prove the stated $L^\I_{loc}$ bound,
 we may assume $\norm{u}_{L^s(0,1;L^q(B_1))}$ is sufficiently small. The general case
follows
by a usual scaling argument.

Recall the system \eqref{v-eq} for the localized velocity $v=\ph u+\nb \eta$,
  \[ \partial_t v - \Delta v + \nabla q = f^0 + \nabla \cdummy f^1 + \nabla
     \cdummy f^2 ( v) ,\quad \div v=0,
\]
  where $f^0, f^1$ and $f^2$ are given in \eqref{f0.def},  \eqref{f1.def} and  \eqref{f2.def}, and $v$ vanishes at $t=0$ and at $|x|=1$.
Consider the map
\[ 
\label{A11}
\Lambda v = v^0 - \Phi ( \tilde{\varphi} u \otimes v) 
\]
where $v_0$ is defined as in \eqref{v0.def}, but
\[
\Phi ( F)_i(x,t) = \int_0^t \int_{\R^3} \partial_k S_{i j} ( x - y, s)
     F_{j k} ( y, t - s) d y d s \]
is defined by the usual convolution, and not by duality as in Lemma \ref{th2.2}.

Our goal is to prove that $\Lambda v$ has a fixed point $v\in
L^{s}L^{q}\cap L^{s}L^{q+\delta}$ for some $\delta>0$, and then to prove it is unique in $L^{s}L^{q}$. We will attain our goal in four steps.

\medskip
{\it{Step 1:}} We show that $v^{0}\in L^{s}(0,1; L^{q}\cap
L^{q+\delta}( \R^3))$  for some $\delta>0$. 

We may assume $m \le s/2$.
 Since $u \in L^s L^q$, $p \in L^m L^1$ and $\varphi$ has compact support, we have
  $p,f^0 \in L^m ( L^1 \cap L^{q / 2})$ by \eqref{better.p}, H\"older's inequality and $f^1 \in L^s
  L^r$ for any $1 \leq r \leq q$ by embedding theorem. 
For $\de=0$	or $0<\de\ll 1$, by Young's inequality and \eqref{estimate-T}, we have
  \begin{eqnarray*}
\label{a2v0}
    \| v^0(t) \|_{q + \delta} & \leq & \int_0^t \| S_{i j} ( s) \|_a \| f^0_j ( t
    - s) \|_{q / 2} d s + \int_0^t \| \partial_k S_{i j} ( s) \|_b \| f^1_j (
    t - s) \|_r d s\\
    & \lec & \int_0^t s^{3 / ( 2 a) - 3 / 2} \| f^0 ( t - s) \|_{q / 2} d s
    + \int_0^t s^{3 / ( 2 b) - 2} \| f^1 ( t - s) \|_r d s
  \end{eqnarray*}
  where $1 / ( q + \delta) = 1 / a + 2 / q - 1 = 1
  / b + 1 / r - 1$, and $1\le a,b<\I$. By Young's inequality, 
\[ 
\| v^0 \|_{L^s L^{q + \delta}} \lec \| f^0  \|_{L^m L^{q / 2}}
     \| s^{ 3 / (2 a) - 3 / 2 }\|_{L^{\rho}} + \| f^1  \|_{L^s L^r}
     \int_{0}^{1} s^{3 / ( 2 b) - 2}ds,
\]
where  $1/s =1/m+1/\rho - 1$. The last integral is finite if  $1 \leq b < 3 / 2$,
which can be achieved by taking $r=q$, and $\de=0$ or $\de> 0$ sufficiently small.
For $\| s^{ 3 / (2 a) - 3 / 2 }\|_{L^{\rho}}$ to be finite,
we need
\[ 
\label{A15}
 ( \frac 3 2 - \frac 3 {2a}) \rho <  1 .
\]
By the relations
\[ 
 \frac{1}{q+\delta} =\frac{1}{a} +\frac{2}{q}-1, \quad \frac{1}{s} = \frac{1}{m} + \frac{1}{\rho} -
     1, \quad \frac{2}{s}+\frac{3}{q}=1 ,
\]
\eqref{A15} is equivalent to
\[ 
\frac{1}{m} < \frac{3}{s} + \frac{3}{2 ( q + \delta)}  , 
\]
which  can be achieved if
\[ 
\frac{1}{m} < \frac{3}{s} + \frac{3}{2 q}= \frac32-\frac3q , 
\]
and by taking  $\de=0$ or $\de> 0$ sufficiently small.

\medskip
{\it{Step 2:}} The cut-off vector $\ph u + \nb \eta$ is a fixed point of 
\eqref{A11} in the class $L^s L^q$. The proof is similar to Lemma
\ref{th3.2}.

\medskip
{\it{Step 3:}} 
We show that $\Phi ( \tilde{\varphi} u \otimes v)\in
L^{s}L^{q+\delta}$ provided $v \in L^s L^{q +
  \delta}$, for either $\de=0$ or $0<\de \ll 1$. Let $r=( q^2 +
       q \delta) / ( 2 q + \delta)$ so that $1/r=1/q+1/(q+\de)$.
By Young's inequality and
(\ref{estimate-T}),
 \begin{align}
\nonumber
       \| \Phi ( \tilde{\varphi} u \otimes v) (t)\|_{L^{q + \delta}} & \leq 
       \int_0^t \| \partial_k S_{i j} ( s) \|_{L^{q'}} \| ( \tilde{\varphi} u \otimes v)_{j k} ( t - s) \|_{L^{r}} d s\\
       & \leq  \int_0^t s^{3 /(2q') - 2}
       \| ( \tilde{\varphi} u \otimes v) ( t - s) \|_{L^r} d s.
  \end{align}
By generalized Young's inequality,
\[ 
\label{a2v2}
\| \Phi ( \tilde{\varphi} u \otimes v) \|_{L^s L^{q + \delta}} \lec
\| s^{3 / (2 q') - 2} \|_{L^{s', \infty}}  \| 
     \tilde{\varphi} u \otimes v\|_{L^{s / 2} L^{r}} \lec \norm{u}_{L^sL^q}
\norm{v}_{L^s L^{q+\de}}.
\]

\medskip  
 {\it Step 4:} Existence in $L^s(L^q\cap L^{q+\de})$ and uniqueness in $L^sL^q$.
They are proved in the same way as Lemmas \ref{th3.3} and \ref{th3.4}.

The above shows that $u$ is locally in $L^sL^{q+\de}$ for some $\de>0$.
By Theorem \ref{thm6.1}, $u$ is locally bounded.
\end{proof}

\section*{Acknowledgments}
Part of this work was done when Tsai visited the Center of Advanced Study 
in Theoretical Sciences (CASTS) at the National Taiwan University. He would like to
thank his hosts for their warm hospitality. The research of Luo and Tsai 
is supported in part by NSERC grant 261356-13 (Canada).

\bibliographystyle{habbrv}
\bibliography{wl3reg}

Yuwen Luo

Department of Mathematics, University of British Columbia,

Vancouver, BC, Canada

ywluo@math.ubc.ca\\

Tai-Peng Tsai

Department of Mathematics, University of British Columbia,

Vancouver, BC, Canada 

ttsai@math.ubc.ca

\end{document}